\documentclass{amsart}
\usepackage{amssymb}
\usepackage{amsfonts}
\usepackage{amssymb}
\usepackage{amsmath}
\usepackage{amsthm}
\usepackage{pdfpages}
\usepackage{enumerate}
\usepackage{tabularx}
\usepackage{centernot}
\usepackage{mathtools}
\usepackage{stmaryrd}
\usepackage{amsthm,amssymb}
\usepackage{etoolbox,color}
\usepackage{tikz}
\usepackage{amssymb}
\usetikzlibrary{matrix}
\usepackage{tikz-cd}
\usepackage{tikz}
\usepackage{marginnote}
\definecolor{mygray}{gray}{0.85}
\usepackage[linecolor=black, backgroundcolor=mygray,colorinlistoftodos,prependcaption,textsize=small]{todonotes}
\usepackage{xargs}  
\usepackage[colorlinks,citecolor=blue,urlcolor=blue, linkcolor=blue]{hyperref}

\renewcommand{\leq}{\leqslant}
\renewcommand{\geq}{\geqslant}

\renewcommand{\trianglelefteq}{\trianglelefteqslant}
\newcommand{\mrm}[1]{\mathrm{#1}}

\usepackage[backgroundcolor=mygray,colorinlistoftodos,prependcaption,textsize=small]{todonotes}
\usepackage{xcolor}

\makeatletter
\def\subsection{\@startsection{subsection}{3}%
  \z@{.5\linespacing\@plus.7\linespacing}{.3\linespacing}%
  {\bfseries\centering}}
\makeatother

\makeatletter
\def\subsubsection{\@startsection{subsubsection}{3}%
  \z@{.5\linespacing\@plus.7\linespacing}{.3\linespacing}%
  {\centering}}
\makeatother

\makeatletter
\def\myfnt{\ifx\protect\@typeset@protect\expandafter\footnote\else\expandafter\@gobble\fi}
\makeatother

\newtheorem{theorem}{Theorem}[section]

\newtheorem{convention}[theorem]{Convention}

\newtheorem{lemma}[theorem]{Lemma}
\newtheorem{proposition}[theorem]{Proposition}

\newtheorem{proviso}[theorem]{Proviso}

\theoremstyle{plain}

\theoremstyle{definition}

\newtheorem{cclaim}[theorem]{Claim}
\newtheorem{fact}[theorem]{Fact}
\newtheorem{mclaim}[theorem]{Main Claim}
\newtheorem{definition}[theorem]{Definition}
\newtheorem{remark}[theorem]{Remark}

\newtheorem{observation}[theorem]{Observation}
\newtheorem{notation}[theorem]{Notation}

\newcounter{claimcounter}

\begin{document}

\begin{abstract} 
An uncountable $\aleph_1$-free group cannot admit a Polish group topology but an uncountable $\aleph_1$-free abelian group can, as witnessed, for example, by the Baer-Specker group $\mathbb{Z}^\omega$; more strongly, $\mathbb{Z}^\omega$ is separable. In this paper we investigate $\aleph_1$-free abelian non-Archimedean Polish groups. We prove two main results. The first is that there are continuum many separable (and so torsionless, and so $\aleph_1$-free) abelian non-Archimedean Polish groups which are pairwise not topologically isomorphic. The second is that the following four properties are complete co-analytic subsets of the space of closed abelian subgroups of $S_\infty$: separability, torsionlessness, $\aleph_1$-freeness and $\mathbb{Z}$-homogeneity.
\end{abstract}

\title{$\aleph_1$-free abelian non-Archimedean Polish groups}


\thanks{No. 1254 on Shelah's publication list. Research of the first author was  supported by project PRIN 2022 ``Models, sets and classifications", prot. 2022TECZJA, and by INdAM Project 2024 (Consolidator grant) ``Groups, Crystals and Classifications''. Research of the second author was partially supported by Israel Science Foundation (ISF) grants no: 1838/19 and 2320/23. The first author would like to thank A. Nies for useful discussions related to this paper.}

\author{Gianluca Paolini}
\address{Department of Mathematics ``Giuseppe Peano'', University of Torino, Via Carlo Alberto 10, 10123, Italy.}
\email{gianluca.paolini@unito.it}

\author{Saharon Shelah}
\address{Einstein Institute of Mathematics,  The Hebrew University of Jerusalem, Israel \and Department of Mathematics,  Rutgers University, U.S.A.}
\email{shelah@math.huji.ac.il}

\date{\today}
\maketitle




\section{Introduction}

Recall that a topological group $G$ is said to be Polish if its group topology is Polish, i.e., separable and completely metrizable, and it is said to be non-Archimedean Polish if, in addition, it admits a countable neighborhood basis at the identity consisting of open subgroups. In a meeting in Durham in 1997, D. Evans asked if a non-Archimedean Polish group can be an uncountable free group. Around the same time H. Becker and A. Kechris asked the same question without the assumption of non-Archimedeanity (cf. \cite{kechris}). In \cite{shelah} the second author settled the first question, and in the subsequent paper \cite{shelah_1} he settled also the second question. It was later discovered that some of the impossibility results from \cite{shelah_1} followed already from an old important result of Dudley \cite{dudley}. All these results were later vastly generalized in \cite{paolini&shelah} by both authors of the present paper in the context of graph products of groups. In another direction, Khelif proved in \cite{khelif} that no uncountable $\aleph_1$-free groups can be Polish, where we recall that a group $G$ is said to be $\aleph_1$-free (or $\aleph_1$-free abelian) if every countable subgroup of $G$ is free (or free abelian, respectively). Furthermore, an uncountable Polish group cannot be free abelian either, by \cite{shelah_1}. In contrast, as well-known, the Baer-Specker group $\mathbb{Z}^\omega$ when topologized with the product topology is non-Archimedean Polish and $\aleph_1$-free abelian (but not free abelian). This interesting state of affairs motivated us to investigate the structure of the class of $\aleph_1$-free abelian non-Archimedean Polish groups, wondering \mbox{in particular: how  complicated is the class of these topological groups?}

\medskip
\noindent

	Recently, there has been significant interest in abelian (non-Archimedean) Polish groups. On one hand, with respect to the study of the Borel complexity of their actions, see e.g. \cite{gao_tame}, and on the other hand, in the context of homological algebra, see e.g. \cite{lupini1, lupini2}. Our study, which is algebraic in nature, complements these existing lines of inquiry by providing new structural algebraic insights into these groups.
	
	
	
%

	\medskip
	
	We recall the following definitions as they will play a crucial role in our study:
	
	\begin{definition}\label{def_intro} Let $A$ be a torsion-free abelian group. 
	\begin{enumerate}[(1)]
	\item We say that $A$ is {\em separable} if every finite subset of $A$ is contained in a free direct summand of $A$.
	\item We say that $A$ is {\em torsionless} if for every $0 \neq a \in A$ there is $f \in \mrm{Hom}(A, \mathbb{Z})$ such that $f(a) \neq 0$.
	\item We say that $A$ is {\em $\mathbb{Z}$-homogenenous} if every element has type $\mathbf{0}$ (cf. \ref{def_type}).
	\end{enumerate}
\end{definition}

It is well-known, see e.g. \cite[Corollary~2.9]{mekler}, that we have the following:

$$\text{separable } \Rightarrow \text{ torsionless } \Rightarrow \text{ $\aleph_1$-free } \Rightarrow \text{$\mathbb{Z}$-homogenenous}.$$

Now, it so happens that $\mathbb{Z}^\omega$ is not only $\aleph_1$-free but actually separable. Thus, the motivating question of the paper mentioned above (on the existence of $\aleph_1$-free abelian non-Archimedean Polish groups) can be strengthened to torsionless, or even separable, abelian groups. We will deal with all these properties in our paper.

\medskip

The first result concerns the number of non-isomorphic groups realizing the four properties above, and it is of the strongest possible form. We recall that a completely decomposable abelian group is a group that is a direct sum of groups of rank~$1$.

\begin{theorem}\label{second_th} There are continuum many separable (hence torsionless, hence {$\aleph_1$-free}) abelian non-Archimedean Polish groups which are pairwise not topologically isomorphic. Furthermore, all these groups can be taken to be inverse limits of countable torsion-free completely decomposable groups.
\end{theorem}

	The second result that we present is formulated in the framework of invariant descriptive set theory (see \cite{gao} for an excellent introduction to this subject). In~\cite{nies}, Kechris, Nies and Tent proposed the following program: as any non-Archimedean Polish group is topologically isomorphic to a closed subgroup of the topological group $S_\infty$ of bijections of $\mathbb{N}$ onto $\mathbb{N}$, we can study classification problems on non-Archimedean Polish groups with respect to the Effros structure on the set of closed subgroups of $S_\infty$, which is a standard Borel space (see \cite{nies} for details). 
	
	\smallskip
	At the heart this program there is the following two-fold task:
	\begin{enumerate}[(1)]
	\item[$\bullet$] for natural classes of closed subgroups of $S_\infty$, determine
whether they are Borel;
	\item[$\bullet$] if a class $\mathcal{C}$ is Borel, study the Borel complexity of
topological isomorphism on~$\mathcal{C}$.
	\end{enumerate}
	
	In this respect, in our second theorem, we prove that the four properties mentioned above are not Borel and, more strongly, that they are complete co-analytic.
	
	\begin{theorem}\label{being_almost_free_coanalytic} Determining if a non-Archimedean Polish group is in $\mathcal{C}$ is a complete co-analytic problem in the space of closed subgroups of $S_\infty$ for the following classes $\mathcal{C}$ of abelian groups:
	\begin{enumerate}[(1)]
	\item $\mathbb{Z}$-homogeneous;
	\item $\aleph_1$-free;
	\item torsionless;
	\item separable.
	\end{enumerate}
\end{theorem}

	We conclude the paper with a result of independent interest:

\begin{theorem}\label{th_ind_interest} There exists an inverse limit of free abelian groups (and so an abelian non-Archimedean Polish group) which is $\mathbb{Z}$-homogenenous but not $\aleph_1$-free.
\end{theorem}

The structure of the paper is as follows. In Section~2 we collect the necessary 
preliminaries. In Section~3 we introduce the main group-theoretic construction 
at the heart of the paper. In Section~4 we prove Theorems~\ref{second_th} 
and~\ref{being_almost_free_coanalytic}. Finally, in Section~5 we establish some 
complementary results.

\section{The descriptive set theory setting}


	From here on, as is customary in set theory, we denote the set of natural numbers by $\omega$. The aim of this section is to fix the coding, i.e., to specify which spaces we refer to in our results, together with introducing the necessary background.

\begin{fact}[{\cite[Proposition~2.1]{gao_tame}}]\label{char_abelian}\label{char_polish} Let $G$ be a topological group, then the following are equivalent:
	\begin{enumerate}[(1)]
	\item $G$ is a non-Archimedean Polish abelian group;
	\item $G$ is isomorphic (as a topological group) to a closed subgroup of $\prod_{n < \omega} H_n$ (with the product topology), with each $H_n$ abelian, countable and discrete;
	\item $G$ is the inverse limit of an inverse system $(G_n, f_{(n, m)} : m \leq n < \omega)$, with each $G_n$ abelian, countable, discrete, \mbox{where on $G$ we consider the product topology.}
\end{enumerate}
\end{fact}

	\begin{notation}\label{space_TFAB_polish} Let $X$ be a Polish space. The Effros structure on $X$ is the Borel space consisting of the family $\mathcal{F}(X)$ of closed subsets of $X$ together with the $\sigma$-algebra generated by the following sets $\mathcal{C}_U$, where, for $U \subseteq X$ open, we let:
	$$\mathcal{C}_U = \{D \in \mathcal{F}(X) : D \cap U \neq \emptyset\}.$$
\end{notation}

	\begin{notation}\label{S_inf_topology} We denote by $S_\infty$ the topological group of bijections of $\omega$ onto $\omega$, where the topology is the one induced by the following metric:
	\begin{enumerate}[(1)]
	\item if $x = y$, then $d(x, y) = 0$;
	\item if $x \neq y$, then $d(x, y) = 2^{-n}$, where $n$ is least such that $x(n) \neq y(n)$.
	\end{enumerate}
\end{notation}

\begin{fact}\label{inv_is_Borel_pre} The closed subgroups of $S_\infty$ form a Borel subset of $\mathcal{F}(S_\infty)$ (see \cite{nies}), which we denote 
by $\mrm{Sgp}(S_\infty)$,
together with the Borel \mbox{structure inherited from $\mathcal{F}(S_\infty)$.}
\end{fact}

	\begin{definition} By a tree on $\omega$ we mean a non-empty subset of $\omega^{<\omega}$ closed under initial segments (so in particular the empty sequence $\emptyset$ is the root of the tree).
\end{definition}

\begin{notation}\label{notation_branches} Given a tree $T \subseteq \omega^{< \omega}$ we denote by $[T]$ the set of $\eta \in \omega^\omega$ such that $\eta \restriction n \in T$ for all $n < \omega$, i.e., the set of infinite branches of $T$.
\end{notation}

	\begin{fact}\label{the_useful_fact} Let $\mathcal{T}(S_\infty)$ denote the set of trees on $\omega$ with no leaves topologized defining the distance between two trees $T_1 \neq T_2$ as $\frac{1}{2^{n}}$, where $n$ is least $n$ such that $T_1 \cap \omega^n \neq T_2 \cap \omega^n$. For $C \in \mrm{Sgp}(S_\infty)$, define $T_C \in \mathcal{T}(S_\infty)$ as follows: 
	$$\mathbf{B}: C \mapsto T_C = \{g \restriction n : g \in C \text{ and } n < \omega\} \cup \{g^{-1} \restriction n : g \in C \text{ and } n < \omega\}.$$
Then $\mathbf{B}: \mrm{Sgp}(S_\infty) \rightarrow \mathcal{T}(S_\infty)$ is Borel, it is injective and it has closed range; furthermore, when restricted to its range, the map $\mathbf{B}$ is a Borel isomorphism.
\end{fact}

	By Fact~\ref{the_useful_fact}, we can work with $\mathbf{B}(\mrm{Sgp}(S_\infty)) \subseteq \mathcal{T}(S_\infty)$ instead of $\mrm{Sgp}(S_\infty)$.


%

	\begin{proviso}\label{the_proviso} From now on we only look at inverse systems with onto bonding  maps.
\end{proviso}

\begin{notation}\label{space_inv} We can look at an inverse system $A = (A_n, f_{(n, m)} : m \leq n < \omega)$ of groups as a first-order structure in a language $L_{\mrm{inv}} = \{P_n, \cdot_n, f_{(n, m)} : m \leq n < \omega \}$, where:
	\begin{enumerate}[(1)]
	\item the $P_n$'s are disjoint predicates denoting the groups $A_n$'s;
	\item $\cdot_n$ is a binary function symbol, interpreted as the group operation on $A_n$;
	\item the $f_{(n, m)}$'s are function symbols interpreted as morphisms from $A_n$ to $A_m$. 
	\end{enumerate}
\end{notation}

	\begin{notation}\label{comment} Of course with respect to the language considered in \ref{space_inv} saying that a structure of the form $A = (A_n, f_{(n, m)} : m \leq n < \omega)$ is an inverse system is axiomatizable in first-order logic. Thus, we can consider the Borel space of countable $\mrm{Pro}$-$\mrm{Groups}$ with domain (a subset of) $\omega$. We also assume that our inverse systems $A = (A_n, f_{(n, m)} : m \leq n < \omega)$ are such that  for $n \neq m$, $\mrm{dom}(A_n) \cap \mrm{dom}(A_m) = \emptyset$. Finally, naturally when the groups $A_n$ are abelian, then we move to additive notation and so we write $+_{n}$ instead of $\cdot_n$ and $0_{A_n}$ instead of $e_{A_n}$.
\end{notation} 

	\begin{definition}\label{comment_embedding} Let $A = (A_n, f_{(n, m)} : m \leq n < \omega)$ be an inverse system as in \ref{comment}. For every $\bar{a} = (a_n : n < \omega) \in \varprojlim(A)$ we define a permutation $\pi_{\bar{a}}$ of $\omega$ as:
	$$\pi_{\bar{a}}(k) = \begin{cases} k +_{A_n} a_n, \;\;\, \text{ if there is } n < \omega \text{ s.t. } k \in A_n, \\
	k, \;\;\;\;\;\;\;\;\;\;\;\;\;\;\,\, \text{ if } k \notin \bigcup_{n < \omega} A_n.\end{cases}$$
\end{definition}

	\begin{observation} The map $\bar{a} \mapsto \pi_{\bar{a}}$ from \ref{comment_embedding} is an embedding of $\varprojlim(A)$ into~$S_\infty$.
\end{observation}

	We omit the details of the next proposition as its proof is standard (use e.g. \ref{the_useful_fact}).

\begin{proposition}\label{inv_is_Borel} The map $A \mapsto \varprojlim(A)$ from $\mrm{Pro}$-$\mrm{Groups}$ (recalling \ref{comment}) into $\mrm{Sgp}(S_\infty)$ (recalling \ref{inv_is_Borel_pre} and the embedding described in \ref{comment_embedding}) is Borel.
\end{proposition}

\section{A few crucial claims on $\mrm{TFAB}$'s}

\subsection{The engine}	The following construction will be used to prove Theorems~\ref{second_th} and \ref{being_almost_free_coanalytic}.

	\begin{notation} $\mathbb{P}$ denotes the set of prime numbers.
\end{notation}

	\begin{definition}\label{the_engine} Let $T \subseteq \omega^{<\omega}$ be a tree with no leaves and $L: T \rightarrow \mathcal{P}(\mathbb{P})$ be such that:
	\begin{enumerate}[(a)]
	\item if $\eta \trianglelefteq \nu$, then $L(\nu) \subseteq L(\eta)$;
	\item if $\eta \in T$, then $L(\eta) = \bigcup \{L(\nu) : \nu \in T \text{ and } \nu \text{ is a successor of $\eta$} \}$.
\end{enumerate}
We define an inverse system $\mrm{inv}(T, L) = (G_n, f_{(m, n)} : m \leq n < \omega)$ as follows:
\begin{enumerate}[(i)]
	\item $Z_n :=\sum \{ x_\eta : \eta \in T \cap \omega^n\} \leq G_n \leq \sum \{\mathbb{Q} x_\eta : \eta \in T \cap \omega^n\} := Q_n$;
	\item $G_n = \langle \frac{1}{p} x_\eta : \eta \in T \cap \omega^n \text{ and } p \in L(\eta) \rangle_{Q_n}$;
	\item for $m \leq n < \omega$, $f_{(n, m)}$ is the unique homomorphism from $G_n$ into $G_m$ s.t.:
	$$\eta \in T \cap \omega^n \; \Rightarrow \; x_\eta \mapsto x_{\eta \restriction m}.$$
\end{enumerate}
Notice that the inverse system $\mrm{inv}(T, L) = (G_n, f_{(n, m)} : m \leq n < \omega)$ is well-defined because of conditions (a) and (b) above. It being well-defined, we do the following:
$$(T, L) \mapsto \mrm{inv}(T, L) \mapsto \varprojlim(\mrm{inv}(T, L)) := G(T, L).$$
So the group $G(T, L)$ is a well-defined Polish non-Archimedean abelian group, when equipped with the topology inherited from the product topology on $\prod_{n < \omega} G_n$. 
\end{definition}


	\begin{remark} Notice that the inverse systems $\mrm{inv}(T, L) = (G_n, f_{(n, m)} : m \leq n < \omega)$ defined in \ref{the_engine} have bonding maps $f_{(n, m)}$ which are onto; this is because we ask that the tree $T$ has no leaves and because we ask that condition (b) from there holds. Recall that in \ref{the_engine} we only consider trees with no leaves and so, for $T$'s as in \ref{the_engine}, the set $[T]$ is always non-empty (cf. \ref{notation_branches} for a definition of $[T]$). 
\end{remark}

	\begin{remark}\label{remark_dir_indec} Notice also that for $\mrm{inv}(T, L) = (G_n, f_{(n, m)} : m \leq n < \omega)$ as in \ref{the_engine} we have that the groups $G_n$ are completely decomposable, i.e., they are direct sums of groups of rank $1$, in fact it follows immediately from the definitions that:
	$$G_n = \sum_{\eta \in \omega^n \cap T} \langle x_\eta \rangle^*_{G_n},$$
where $\langle X \rangle^*_{G_n}$ denotes pure closure (cf. \cite[pg. 151]{fuchs_book}).
\end{remark}


	For background on the notions that we now introduce see e.g. \cite[Chapter~13]{fuchs_book}. We will often abbreviate ``torsion-free abelian group'' with TFAB.

	\begin{definition} Let $A \in \mrm{TFAB}$ and let $(p_i : i < \omega)$ be the list of the prime numbers in increasing order. For $a \in A$, we define the characteristic of $a$ in $A$, denoted as $\chi_A(a) = \chi(a)$, as follows:
	$\chi(a) = (h_{p_i}(a) : i < \omega),$
where $h_{p_i}(a)$ is the supremum of the $k < \omega$ such that $p_i^k \, \vert \, a$, where the supremum is taken in $\omega \cup \{\infty\}$ (so the value $\infty$ is allowed).
\end{definition}

	\begin{definition}\label{def_type} \begin{enumerate}[(1)]
	\item Two characteristics $(k_i : i < \omega)$ and $(\ell_i : i < \omega)$ are said
to be equivalent if $k_i = \ell_i$ for almost all $i < \omega$ and both $k_i$ and $\ell_i$ are finite whenever $k_i \neq \ell_i$. The equivalence classes of characteristics are called {\em types}. 
	\item Given $A \in \mrm{TFAB}$ we denote by $\mathbf{t}_A(a)$ the type of $a$ in $A$; we may simply write $\mathbf{t}(a)$ when $A$ is clear from the context.
	\item We denote by $\mathbf{0}$ the type of the characteristic $(0, 0, ...)$.
	\item For $A \in \mrm{TFAB}$ and $a, b \in A$ we say that $\mathbf{t}(a) \leq \mathbf{t}(b)$ if there are characteristics $\chi = (k_i : i < \omega) \in \mathbf{t}(a)$ and $\nu = (\ell_i : i < \omega) \in \mathbf{t}(b)$ such that $(k_i : i < \omega) \leq (\ell_i : i < \omega)$, where the order is the pointwise order, i.e., $k_i \leq \ell_i$, for all $i < \omega$.
\end{enumerate} 
\end{definition}

	We will need the following well-known facts.
	
	\begin{fact}[{\cite[(C) on pg. 411]{fuchs_book}}]\label{Fuchs_fact} Let $G$ and $H$ be torsion-free abelian and suppose that there is a homomorphism $f: G \rightarrow H$, then, for every $g \in G$, $\mathbf{t}(g) \leq \mathbf{t}(f(g))$.
\end{fact}

	\begin{fact}{\cite[Theorem~2.3]{mekler}}\label{first_theorem_copy} Let $A \in \mrm{TFAB}$. Then $A$ is $\aleph_1$-free if and only if, for every finite $S \subseteq A$, the pure closure of $S$ in $A$ (denoted $\langle S \rangle^*_A$) is free abelian.
\end{fact}

%


\subsection{Nice pairs}\label{sec_nice_pairs}

	We now present a technical definition necessary for our approach to work, the fact that this or similar technical adjustments are necessary is justified by Section~\ref{final_subsection}.

	\begin{definition}\label{def_nice} We call $(T, L)$ a nice pair when for some $S$ and $\bar{p}$ we have that:
	\begin{enumerate}[(a)]
    \item $S$ is a subtree of $T$ and $T$ is a subtree of $\omega^{< \omega}$ with no leaves;
    \item\label{def_nice_item2} $\eta \in S$ implies that $0 \notin \mrm{ran}(\eta)$;
    \item $\bar{p} = (p_\eta : \eta \in S)$ is a sequence of primes without repetitions;
    \item $S^+ = \{ \eta^\frown\nu : \eta \in S, \; \nu \in \{0\}^{< \omega}\} \subseteq T$;
    \item $p_{\eta^\frown\nu} = p_\eta$, when $\eta \in S$ and $\nu \in \{0\}^{< \omega}$;
    \item $L$ is the function with domain $T$ defined as follows:
  $$L(\eta) = \{ p_\nu : (\nu \triangleleft \eta \text{ or } \eta \trianglelefteq \nu) \text{ and } \nu \in S^+\}.$$
\end{enumerate}
\end{definition}

	\begin{observation}\label{crucial_obs} If $(T, L)$ a nice pair, then $(T, L)$ is as in \ref{the_engine}.
\end{observation}

	\begin{convention}\label{the_convention} For the rest of this subsection, we fix $(T,L)$, $S$, $S^+$ and $\bar{p}$ as in \ref{def_nice} and we let $G = G(T, L)$. For $\eta \in [T]$, we let $L(\eta) = \bigcap \{L(\eta \restriction n) :  n < \omega \}$. Also, for $n < \omega$, we let $f_{(\omega, n)}$ be the obvious projection of $G$ onto $G_n$. For $\nu \in [T]$, we let $x_{\nu} = (x_{\nu \restriction n} : n < \omega) \in G$. Recall also that $f_{(n, m)}$ denote the bonding map from $G_n$ onto $G_m$.
\end{convention}

	\begin{cclaim}\label{inf_branch_claim} If $S$ has an $\omega$-branch (i.e., there is $\eta \in [T]$ such that, for all $n < \omega$, $\eta \restriction n \in S$), then $G$ is not \mbox{$\mathbb{Z}$-homogeneous, and so $G$ is not $\aleph_1$-free.}
\end{cclaim}

	\begin{proof} In this case $L(\eta) = \bigcap \{L(\eta \restriction n) : n < \omega \} = \{p_{\eta \restriction n} : n < \omega\}$ is infinite, and so in $G$ we have that $x_\eta$ is divisible by $p$ for every $p \in L(\eta)$.
	\end{proof}
	
	\begin{definition}\label{Keta} For $\eta \in T \cup [T]$, let $K_\eta$ be the additive subgroup of $\mathbb{Q}$ generated by $\{\frac{1}{p} : p \in L(\eta)\} \subseteq \mathbb{Q}$; explicitly, $K_\eta$ is the subgroup of $\mathbb{Q}$ consisting of:
	$$\{\frac{b}{\prod_{\ell < k}p_\ell} : b \in \mathbb{Z}, \; k < \omega, \; p_\ell \in L(\eta)\}.$$
\end{definition}
	
	\begin{mclaim}\label{the_hammer} If $S$ has no $\omega$-branch, then $G = G(T, L)$ is separable (and so $G$ is torsionless, and so $G$ is $\aleph_1$-free).
\end{mclaim}

	\begin{proof} 

\noindent We show that $G$ is separable. To this end, let $X \subseteq G$ be finite.
\begin{enumerate}[$(*_1)$]
	\item \underline{We claim that it suffices to prove} that there are $n < \omega$ and finite $Y \subseteq \omega^n \cap T$ such that letting $\pi_{(n, Y)}$ be the natural projection of $G_n$ onto $G_{(n, Y)} := \langle Y \rangle^*_{G_n}$ we have:
	\begin{enumerate}[(a)]
	\item for every $\nu \in Y$, $L(\nu)$ is finite;
	\item $g = \pi_{(n, Y)} \circ f_{(\omega, n)}$;
	\item $g \in \mrm{Hom}(G, G_{(n, Y)})$;
	\item $g_0 := g \restriction \langle X \rangle^*_G$ is an isomorphism from $\langle X \rangle^*_G$ onto $G_{(n, Y)} = \langle Y \rangle^*_{G_n}$.
	\end{enumerate}
\end{enumerate}
We show that $(*_1)$ holds, i.e., that the conditions listed there suffice to conclude that $G$ is separable. For this purpose, first of all, notice that $G_{(n, Y)}$ is a free direct summand of $G$, because $G_n = \bigoplus \{K_\nu x_\nu : \nu \in T \cap \omega^n \}$ and, for every $\nu \in Y$, $L(\nu)$ is finite and so $K_\nu \cong \mathbb{Z}$. Secondly, observe that then also $\langle X \rangle^*_G$ is free, since by (c), $g_0$ is an isomorphism. Observe now that, by clause (b), we have that $g_0^{-1} \circ g$ is an homomorphism from $G$ into $\langle X \rangle^*_G$. Also, by clause (c), clearly $g_0^{-1}$ is an isomorphism from $G_{(n, Y)}$ onto $\langle X \rangle^*_G$. Notice now that for $a \in \langle X \rangle^*_G$ we have that:
$$g_0^{-1} \circ g(a) = g_0^{-1}(g(a)) = g_0^{-1}(g_0^{-1}(a)) = a.$$
Hence, we have that $g_0^{-1} \circ g$ is a retraction of $G$ onto the free subgroup $\langle X \rangle^*_G$ and so $X$ is contained in a free direct summand of $G$, as desired. Hence, $(*_{1})$ holds indeed.

\medskip \noindent So now that we have showed that $(*_{1})$ suffices, for the rest of the proof we will focus on showing that the conditions from $(*_{1})$ can be fulfilled (cf. the final $(*_{19})$).
\begin{enumerate}[$(*_{2})$]
	\item For every $n < \omega$, $f_{(n+1, n)}$ maps $f_{(\omega, n+1)}(\langle X \rangle^*_G)$ onto $f_{(\omega, n)}(\langle X \rangle^*_G)$, hence:
	\begin{enumerate}[$(\cdot_1)$]
	\item $\mrm{rk}(f_{(\omega, n+1)}(\langle X \rangle^*_G)) \geq \mrm{rk}(f_{(\omega, n)}(\langle X \rangle^*_G))$;
	\item those ranks are $\leq |X|$ which is finite.
	\end{enumerate}
	\end{enumerate}
	\begin{enumerate}[$(*_{3})$]
	\item For some $n_* < \omega$, $( \mrm{rk}(f_{(\omega, m)}(\langle X \rangle^*_G)) : m \geq n_*)$  is constant, call it $k_*$, and for every $n_2 > n_1 \geq n_*$, $f_{(n_2, n_1)} \restriction f_{(\omega, n_2)}(\langle X \rangle^*_G)$ is an isomorphism from $f_{(\omega, n_2)}(\langle X \rangle^*_G)$ onto $f_{(\omega, n_1)}(\langle X \rangle^*_G)$.
	\end{enumerate}
[Why? As $n_2 > n_1 \geq n_*$ and so $\mrm{rk}(f_{(\omega, n_2)}(\langle X \rangle^*_G)) =\mrm{rk}(f_{(\omega, n_1)(\langle X \rangle^*_G)})$, recalling that $f_{(n_2, n)}$ maps $f_{(\omega, n_2)}(\langle X \rangle^*_G)$ onto $f_{(\omega, n_1)}(\langle X \rangle^*_G)$, we have that the map $f_{(n_2, n_1)} \restriction f_{(\omega, n_2)}(\langle X \rangle^*_G)$ has trivial kernel and so it is $1$-to-$1$, and trivially it is also onto.]
\begin{enumerate}[$(*_{4})$]
	\item If $n \geq n_*$, then $f_{(\omega, n)} \restriction \langle X \rangle^*_G$ is an isomorphism from $\langle X \rangle^*_G$ onto $f_{(\omega, n)}(\langle X \rangle^*_G)$.
\end{enumerate}
Why $(*_{4})$ holds? Let $n_* \leq n < \omega$ and suppose the negation of the thesis. Then there is $y \in \langle X \rangle^*_G \setminus \{ 0 \}$ such that $f_{(\omega, n)}(y) = 0$. Let $y = (y_m : m < \omega)$. Then we can find $k \geq n$ such that $f_{(\omega, k)}(y) = y_k \neq 0$. But then $y_k \in G_k \setminus \{0\}$ and $f_{(k, n)}(y_k) = 0$, and so $f_{(k, n)}$ has non-trivial kernel, contradicting the choice of $n_*$ (recall that by hypothesis $n \geq n_*$ and that $f_{(k, n)}: G_k \rightarrow G_n$ is onto $G_n$). So $(*_{4})$ holds indeed.
\begin{enumerate}[$(*_{5})$]
	\item W.l.o.g. $|X| = k_*$, equivalently, $X$ is independent.
\end{enumerate}
\begin{enumerate}[$(*_{6})$]
	\item For $n < \omega$.
	\begin{enumerate}[(a)]
	\item Let $Y_n \subseteq T \cap \omega^n$ be finite and such that we have:
	$$f_{(\omega, n)}(\langle X \rangle^*_G) \subseteq G_{(n, Y_n)} = \bigoplus \{ K_\nu x_\nu : \nu \in Y_n\}.$$
	\item For $y \in X$, let $f_{(\omega, n)}(y) = \sum \{a_{(y, n, \nu)} x_\nu : \nu \in Y_n\}$, where $a_{(y, n, \nu)} \in K_\nu$.
	\item  W.l.o.g. for every $n \geq n_*$ and $\nu \in Y_n$, there is $y \in X$ such that $a_{(y, n, \nu)} \neq 0$.
\end{enumerate}	
\end{enumerate}
[Why? Obvious.]
\begin{enumerate}[$(*_{7})$]
	\item Let $Y^\star_n = \{\nu \in Y_n : L(\nu) \text{ is finite} \}$.
\end{enumerate}
\begin{enumerate}[$(*_{8})$]
	\item Recalling the definition of $\pi_{(n, Y^\star_n)}$ from $(*_{1})$, we have that the sequence
	$$(\mrm{rk}(\pi_{(n, Y^\star_n)} \circ f_{(\omega, n)}(\langle X\rangle^*_G)) : n \geq n_*)$$
is eventually constant, say for $n \geq n_{\star}$, and let $k_{\star}$ be this constant value.
\end{enumerate}
\begin{enumerate}[$(*_{9})$]
	\item Assume $n_* \leq n_1 \leq n_2 < \omega$, then:
	\begin{enumerate}[$(\cdot_1)$]
	\item if $\nu \in Y_{n_1}$, then there is $\rho \in Y_{n_2}$ such that $\nu \triangleleft \rho$;
	\item if $\nu \in Y^\star_{n_1}$ and $\nu \triangleleft \rho \in Y_{n_2}$, then $\rho \in Y^\star_{n_2}$ ($Y^\star_{n_1}$ was defined in $(*_{7})$). 
	\end{enumerate}
	\end{enumerate}
[Why $(*_{9})(\cdot_1)$? By $(*_{6})$(c).]
\begin{enumerate}[$(*_{10})$]
	\item W.l.o.g. $n_\star = n_*$ (we shall use this freely).
\end{enumerate}
\begin{enumerate}[$(*_{11})$]
	\item Let $(y_\ell : \ell < k_*)$ be a list of $|X|$ (recall $(*_{5})$).
\end{enumerate}
\begin{enumerate}[$(*_{12})$]
	\item For every $n < \omega$ and $\ell < k_*$, let:
	$$y_{(\ell, n)} = \pi_{(n, Y^\star_n)} \circ f_{(\omega, n)}(y_\ell),$$
\end{enumerate}
where $\pi_{(n, Y^\star_n)}: G_n \rightarrow \langle Y^\star_n\rangle^*_{G_n}$ is the obvious projection onto $\langle Y^\star_n \rangle^*_{G_n}$ (as in $(*_1)$).
\begin{enumerate}[$(*_{13})$]
	\item Hence, we have:
	$$y_{(\ell, n)} = \sum \{a_{(y_\ell, n, \nu)} x_\nu : \nu \in Y^\star_n \}.$$
\end{enumerate}
	\begin{enumerate}[$(*_{14})$]
	\item For $n \geq m \geq n_*$, let:
	\begin{enumerate}[(a)]
	\item $Y^\star_{(n, m)} = \{\nu \in Y^\star_n : \nu \restriction m \in Y^\star_m \}$;
	\item $y_{(\ell, n, m)} =\sum \{a_{(y_\ell, n, \nu)} x_\nu : \nu \in Y^\star_{(n, m)}\}$.
	\end{enumerate}
\end{enumerate}
	\begin{enumerate}[$(*_{15})$]
	\item For $n \geq m \geq n_*$, $y_{(\ell, n, m)} = \pi_{(n, Y^\star_{(n, m)} )}(y_\ell)$.
	\end{enumerate}
	\begin{enumerate}[$(*_{16})$]
	\item $k_{\star} = k_*(= |X|)$ (recall $(*_{5})$ and $(*_{8})$).
\end{enumerate}
We prove $(*_{16})$. Toward contradiction, assume that $k_{\star} < k_*$. For $n \geq n_*$, the sequence $(y_{(\ell, n)} : \ell \leq k_\star)$ cannot be independent, so for some $\ell(n) \leq k_\star$ the element $y_{(\ell(n), n)}$ depends on $\langle  y_{(\ell, n)} : \ell \neq \ell(n), \; \ell \leq k_\star \rangle_{G_n}$. But if $n_2 \geq n_1 \geq n_*$, then using $f_{(n_2, n_1)}$ and $y_{(\ell, n_2)}$, $y_{(\ell, n_2, n_1)}$ we get that $\ell(n_2)$ can serve as $\ell(n_1)$, and so w.l.o.g. $\ell(n) = \ell(n_*)$ for all $n \geq n_*$, so renaming w.l.o.g. we have that $\ell(n) = k_\star$.
\begin{enumerate}[$(*_{16.1})$]
	\item Let $\sum \{b^n_\ell y_{(\ell, n)} : \ell \leq k_\star \} = 0$, where, for $\ell \leq k_\star$, $b^n_\ell \in \mathbb{Z}$ and $b^n_{k_\star} \neq 0$.
	\end{enumerate}
\begin{enumerate}[$(*_{16.2})$]
	\item W.l.o.g. for every $n_2 \geq n_1 \geq n_*$ we have that, for every $\ell \leq k_\star$, $b^{n_2}_{\ell} = b^{n_1}_{\ell}$, call this number $b_\ell$, so that in particular $b_{k_\star} \neq 0$.
\end{enumerate}
\begin{enumerate}[$(*_{16.3})$] 
\item However, in $G$ we have that $z := \sum \{b_\ell y_\ell : \ell \leq k_\star\} \neq 0$ and $z \notin \langle y_\ell : \ell < k_* \rangle^*_{G}$, so there is $n_+ > n_*$ such that $f_{(\omega, n_+)}(z) \neq 0$, so we can find $\nu \in Y_{n_+}$ such that $a_{(z, n_+, \nu)} \neq 0$.
\end{enumerate}
\begin{enumerate}[$(*_{16.4})$] 
\item $\nu \not \in Y^\star_{n_+}$.
\end{enumerate}
[Why? Because $\pi_{(n_+, Y^\star_{n_+})} \circ f_{(\omega, n_+)}(z) = \sum \{b^n_\ell y_{(\ell, n)} : \ell \leq k_\star \} = 0$ by $(*_{16.1})$.]
\begin{enumerate}[$(*_{16.5})$] 
\item There are $n_\dagger$ and $\rho$ such that:
	\begin{enumerate}[(a)]
	\item $n_\dagger \geq n_+ \geq n_*$;
	\item $\rho \in Y_{n_\dagger}$ and $\nu \triangleleft \rho$ (where $\nu$ is as in $(*_{16.3})$, $(*_{16.4})$);
	\item $\rho \in Y^\star_{n_\dagger}$;
	\item $a_{(z, n_\dagger, \rho)} \neq 0$.
	\end{enumerate}
\end{enumerate}
Why does $(*_{16.5})$ hold? We choose $\nu_n$ by induction on $n \geq n_+$ such that:
\begin{enumerate}[(a)]
	\item $\nu_{n_+} = \nu \in Y_{n_+}$;
	\item $\nu_n \in Y_n$;
	\item if $n = m+1 > n_+$, then $\nu_m \triangleleft \nu_n$;
	\item $a_{(z, n, \nu_n)} \neq 0$.
\end{enumerate}
As in the beginning of the proof we can carry the induction. Now, as we are assuming that $S$ has no $\omega$-branch, clearly for some $n_\dagger > n_+$ we have that $\nu_{n_\dagger} \notin S$, but then necessarily $L(\nu_{n_\dagger})$ has to be finite, i.e., $\nu_n \in Y^\star_{n_\dagger}$, since by assumption $S$ is a subtree of $T$ and we defined, for $\eta \in T$,
$L(\eta) = \{ p_\nu : (\nu \triangleleft \eta \text{ or } \eta \trianglelefteq \nu) \text{ and } \nu \in S^+\}$.
\begin{enumerate}[$(*_{16.6})$] 
	\item  $\pi_{(n_\dagger, Y^\star_{n_\dagger})} \circ f_{(\omega, n_\dagger)}(z) \neq 0$.
\end{enumerate}
[Why? By $(*_{16.5})$(d).]
\begin{enumerate}[$(*_{16.7})$] 
	\item $\pi_{(n_\dagger, Y^\star_{n_\dagger})} \circ f_{(\omega, n_\dagger)}(z) = 0$.
\end{enumerate}
[Why? Notice that:
$$\begin{array}{rcl}
 \pi_{(n_\dagger, Y^\star_{n_\dagger})} \circ f_{(\omega, n_\dagger)}(z)
 & = & \pi_{(n_\dagger, Y^\star_{n_\dagger})} \circ f_{(\omega, n_\dagger)}(\sum \{b_\ell y_\ell : \ell \leq k_\star\}) \\
 & = & \sum \{b^{n_\dagger}_\ell y_{(n_\dagger, \ell)} : \ell \leq k_\star\} \\
 & = & \sum \{b_{\ell} y_{(n_\dagger, \ell)} : \ell \leq k_\star\}\\
 & = & 0 \\
\end{array}$$
where the first equality is by the definition of $z$ in $(*_{16.3})$; the second equality is by $(*_{12})$, the third equality is by $(*_{16.2})$ and the fourth equality is by $(*_{16.1})$.]

\smallskip \noindent
But $(*_{16.7})$ contradicts $(*_{16.6})$. Hence, $(*_{16})$ holds indeed.
\begin{enumerate}[$(*_{17})$] 
\item
\begin{enumerate}
	\item Let $f^\star_{(\omega, n)} = \pi_{(n, Y^\star_{n})} \circ f_{(\omega, n)}$;
	\item $f^\star_{(\omega, n)}$ is a homomorphism from $G$ into $\langle Y^\star_{n} \rangle^*_{G_n} \leq G_n$.
\end{enumerate}\end{enumerate}
\begin{enumerate}[$(*_{18})$] 
\item If $n \geq n_*$, then $g_0 = f^\star_{(\omega, n)} \restriction \langle X \rangle^*_G$ is $1$-to-$1$ from $\langle X \rangle^*_G$ onto $\langle Y^\star_n \rangle^*_{G_n}$.
\end{enumerate}
[Why? By $(*_{16})$ and $(*_{17})$, as in the proof of $(*_{4})$.]
\begin{enumerate}[$(*_{19})$] 
\item The conditions from $(*_{1})$ holds and so the proof is complete.
\end{enumerate}
[Why? Let $n \geq n_*$ and let $Y$ from $(*_{1})$ be $Y^\star_n$, then by $(*_{18})$ we are done.]
\end{proof}
	
	\section{Proof of Main Theorems}

	In this section we show how to use our machinery to prove Theorems~\ref{second_th}, 
\ref{being_almost_free_coanalytic}.

	\subsection{Proof of Theorem~\ref{second_th}}
	
	\begin{definition}\label{the_group_G_P} Let $S$ be the following well-founded tree. For $n < \omega$, let $T_n$ be a copy of $\omega^{<n}$ such that for $ n < m < \omega$ we have that $T_n \cap T_m = \emptyset$. Let $S$ be the tree obtained adding a minimum $x_*$ to $\bigcup \{ T_n  : n < \omega\}$, so $x_*$ is the root of the tree $S$. Easily, $S$ can be realized as a subtree of $\omega^{< \omega}$ satisfying condition \ref{def_nice}(2), and $S$ is well-founded. Fix an enumeration of $S$ of order type $\omega$, say $(\eta_i : i < \omega)$. Let $\mathbb{P}_*$ be an infinite set of primes, enumerate $\mathbb{P}_*$ in increasing order as $(p^{\mathbb{P}_*}_i : i < \omega)$ and let $p_{\eta_i} = p^{\mathbb{P}_*}_i$. Then $S$ and $\bar{p} = (p_{\eta_i} : i < \omega)$ are as in Definition~\ref{def_nice}. W.l.o.g. we assume that if $\eta$ is not the root of the tree, then $\eta(0) = n$ if and only if $\eta \in T_n$.
\end{definition}

	\begin{observation} By \ref{the_hammer}, for $(\omega^{< \omega}, L_{\mathbb{P}_*})$ as in \ref{the_group_G_P}, $G(\omega^{< \omega}, L_{\mathbb{P}_*})$ is $\aleph_1$-free.
\end{observation}

	\begin{lemma}\label{when_iso} Let $\mathbb{P}_1$ and $\mathbb{P}_2$ be infinite sets of primes such that $\mathbb{P}_1 \cap \mathbb{P}_2$ is finite, and let $(\omega^{<\omega}, L_{\mathbb{P}_\ell})$ be as in \ref{the_group_G_P} and $G_{\ell} = G(\omega^{<\omega}, L_{\mathbb{P}_\ell})$ (recall \ref{the_engine}). Then there is no topological isomorphism from $G_{1}$ onto $G_{2}$.
\end{lemma}

	\begin{proof} Let $A = G_1 = \varprojlim(G_{(1, n)})$ and  $B = G_2 = \varprojlim(G_{(2, n)})$, where the $G_{(\ell, n)}$'s are from the inverse system $\mrm{inv}(\omega^{<\omega}, L_{\mathbb{P}_\ell})$, of course. Let $G_{(1, n)} = A_n$ and $G_{(2, n)} = B_n$. Also, for $n \leq m < \omega$, denote the bonding map from $A_n$ onto $A_m$ by $t_{(n, m)}$, and the bonding map from $B_n$ onto $B_m$ by $s_{(n, m)}$. Suppose that there is a topological isomorphism from $A$ onto $B$, then, there are $\phi, \psi$, $f_n$ and $g_n$, for $n < \omega$, such that:
	\begin{enumerate}[(a)]
	\item $\phi: \omega \rightarrow \omega$ is increasing;
	\item $\psi: \omega \rightarrow \omega$ is increasing;
	\item $f_n: A_{\phi(n)} \rightarrow B_n$;
	\item $g_n: B_{\psi(n)} \rightarrow A_n$;
	\item for every $n \leq k < \omega$, the following diagram is commutative:
	\begin{center}
		\begin{tikzcd}
			A_{\phi(k)} \arrow[r, "f_k"] \arrow[d, "t_{(\phi(k), \phi(n))}"'] & B_k \arrow[d, "s_{(k,n)}"] \\
			A_{\phi(n)} \arrow[r, "f_n"] & B_n  
		\end{tikzcd}
	\end{center}
	\item for every $n \leq k < \omega$, the following diagram is commutative:
	\begin{center}
		\begin{tikzcd}
			B_{\psi(k)} \arrow[r, "g_k"] \arrow[d, "s_{(\psi(k), \psi(n))}"'] & A_k \arrow[d, "t_{(k,n)}"] \\
			B_{\phi(n)} \arrow[r, "g_n"] & A_n  
		\end{tikzcd}
	\end{center}
	\item for every $n < \omega$, $f_n \circ g_{\phi(n)}: B_{\psi(\phi(n))} \rightarrow B_n$ is equal to $s_{(\psi(\phi(n)), n)}$;
	\item for every $n < \omega$, $g_n \circ f_{\psi(n)}: A_{\phi(\psi(n))} \rightarrow A_n$ is equal to $t_{(\phi(\psi(n)), n)}$.
	\end{enumerate}
\begin{enumerate}[$(*_1)$]
	\item If $\eta \in \omega^{\phi(k)+1}$ is not maximal in $S$, then $f_k(x_\eta) = 0_{B_k}$.
\end{enumerate}
[Why $(*_1)$? By our choice $\eta$ is not maximal in $S$, hence we have that:
$$A_{\phi(k)+1} \models \text{``$x_\eta$ is divisible by infinitely many primes in $\mathbb{P}_1$''},$$
but since $\mathbb{P}_1 \cap \mathbb{P}_2$ is finite, necessarily $f_k(x_\eta) = 0_{B_k}$, by \ref{Fuchs_fact}.]

\smallskip \noindent Let $n = \phi(\psi(0))$ and let $\eta \in \omega^{n+1}$ be such that $\eta$ is not maximal in the tree $S$. Then by $(*_1)$ we have that $f_{\psi(n)} (x_\eta) = 0_{B_{\psi(0)}}$ and so necessarily $g_0 \circ f_{\psi(0)} (x_\eta) = 0_{A_0}$, but this contradicts item (h) above, in fact, since $g_n \circ f_{\psi(n)}: A_{\phi(\psi(0))} \rightarrow A_0$ is equal to $t_{(\phi(\psi(0)), 0)}$ we also have that $g_0 \circ f_{\psi(0)}(x_\eta) = x_{\emptyset}$ and $x_{\emptyset} \neq 0_{A_0}$, a contradiction.
\end{proof}

	\begin{proof}[Proof of \ref{second_th}] Relying on \ref{when_iso}, it suffices to take a collection $\{\mathbb{P}_\alpha : \alpha < 2^{\aleph_0} \}$ of almost disjoint subsets of the set of prime numbers, where $A, B$ are almost disjoint if $A \cap B$ is finite. 
\end{proof}

\subsection{Proof of Theorem~\ref{being_almost_free_coanalytic}}
		
\begin{definition}\label{the_group_G_T} Let $A \subseteq \omega^{< \omega}$ be an infinite tree (in this case, crucially, possibly with leaves). For $\eta \in A$, let $\eta^{+1} = (\eta(n) + 1 : n < \omega)$ and let $A^{+1} = S = \{\eta^{+1} : \eta \in A\}$. Let $\bar{p} = \{p_\eta : \eta \in S\}$ and $T = \omega^{< \omega}$. And let $G(\omega^{< \omega}, L_A)$ be defined as in \ref{def_nice} for this choice of $S$ and $\bar{p}$.
\end{definition}


	\begin{lemma}\label{lemma_coanalytic} For $(\omega^{< \omega}, L_A)$ as in \ref{the_group_G_T}, letting $G_A = G(\omega^{< \omega}, L_A)$ we have:
	\begin{enumerate}[(1)]
	\item if $A$ is well-founded, then $G_A$ is separable, and so torsionless, and so $\aleph_1$-free, and so $\mathbb{Z}$-homogenenous;
	\item if $A$ is not well-founded, then $G_A$ is not $\mathbb{Z}$-homogenenous, and so not $\aleph_1$-free, and so not torsionless, and so not separable.
\end{enumerate}
\end{lemma}

	\begin{proof} This is by \ref{inf_branch_claim} and \ref{the_hammer}.
	\end{proof}




	\begin{proof}[Proof of \ref{being_almost_free_coanalytic}] First of all, recall that, 	as well-known, well-founded trees form a complete co-analytic set. This is by \ref{inv_is_Borel} and \ref{lemma_coanalytic}, together with the observation that the map $A \mapsto \mrm{inv}(\omega^{< \omega}, L_A)$ is Borel, when considering models whose domain is $\subseteq \omega$.
	\end{proof}
\section{Other results}

\subsection{Separating $\mathbb{Z}$-homogeneous and $\aleph_1$-free}

In this section we prove Theorem~\ref{th_ind_interest}. This result is of independent interest.

	\begin{fact}\label{the_groups_for_counterex} For every $2 \leq \mathbf{n} < \omega$ and $\mathbb{P}_*$ an infinite set of primes, there are $K_\mathbf{n}$, $\bar{x}_{\mathbf{n}} = (x_i : i < \mathbf{n})$ and $\bar{y}_{\mathbf{n}} = (y_p : p \in \mathbb{P}_*)$ such that:
	\begin{enumerate}[(1)]
	\item $Z_\mathbf{n} := \mathbb{Z}^\mathbf{n} \leq K_n \leq \mathbb{Q}^\mathbf{n} := H_\mathbf{n}$;
	\item for every $p \in \mathbb{P}_*$, there is $y_p \in K_\mathbf{n} \setminus Z_\mathbf{n}$ such that $p y_p \in H_\mathbf{n}$;
	\item for every $X \subseteq K_\mathbf{n}$ of size $< \mathbf{n}$ we have that $\langle X \rangle^*_{K_\mathbf{n}}$ is free;
	\item $\langle \bar{x} \rangle^*_{K_\mathbf{n}}$ is not free, in fact, more strongly, if $\mathbb{P}_0 \subseteq \mathbb{P}_*$, then the group
	$$K_{(\mathbf{n}, \mathbb{P}_0)} = \langle x_\ell, y_p : \ell < \mathbf{n} \text{ and } p \in \mathbb{P}_0 \rangle_{K_{\mathbf{n}}} \leq \langle \bar{x} \rangle^*_{K_\mathbf{n}}$$
is not free and disjoint from $\{y_p : p \in \mathbb{P}_* \setminus \mathbb{P}_0\}$.
	\end{enumerate}
\end{fact}

	\begin{cclaim}\label{count_1} Let $G = \prod_{n < \omega} K_n$. Then $G$ is $\mathbb{Z}$-homogenenous but not $\aleph_1$-free.
\end{cclaim}

	\begin{proof} Easy. 
\end{proof}
	
	Crucially the group $G$ from \ref{count_1} is not an inverse limit of completely decomposable groups, we want to strengthen \ref{count_1} so as to have in addition this.
	
	\begin{convention}
	\begin{enumerate}[(1)]
	\item $\mathbf{n} \in [2, \omega)$;
	\item $K_\mathbf{n}, H_\mathbf{n}, Z_\mathbf{n}$ are as in \ref{the_groups_for_counterex};
	\item $(p_n : n < \omega)$ lists $\mathbb{P}_\star$ in increasing order.
	\end{enumerate}
\end{convention}

	\begin{definition}\label{def_Gstar} For $n < \omega$, let $G^*_n = \langle x_\ell, y_{p_i} : \ell < \mathbf{n} \text{ and } i < n \rangle_{K_{\mathbf{n}}} \leq K_{\mathbf{n}} \leq \langle \bar{x} \rangle^*_{H_\mathbf{n}}$.
\end{definition}

	\begin{cclaim}\label{G*n_free}
	\begin{enumerate}[(1)]
	\item For every $n < \omega$, $G^*_n$ is free and $G^*_n \leq H_{\mathbf{n}}$;
	\item For every $n < \omega$, $G^*_n$ is completely decomposable;
	\item $(G^*_n : n < \omega)$ is $\leq$-decreasing with intersection $K_{\mathbf{n}}$.
	\end{enumerate}
\end{cclaim}

\begin{proof} It suffices to consider the following group:
$$G^\star_n = \langle \frac{1}{\prod_{i < n} p_\ell} x_0, ..., \frac{1}{\prod_{i < n} p_\ell} x_{\mathbf{n}-1}  \rangle_{H_{\mathbf{n}}}$$
and to observe that $G^*_n \leq G^\star_n$ and that $G^\star_n$ is free.
\end{proof}

	\begin{cclaim}\label{def_inv_comple_dec_count} There is an inverse system $(G_n, f_{(n, m)} : m \leq n < \omega)$ such that:
	\begin{enumerate}[(a)]
	\item for every $n < \omega$, $G_n = \bigoplus \{G_{(n, j)} : j < \omega\}$;
	\item $G_{(n, 0)} = G^*_n$ and, for $j < \omega$,  $G_{(n, j+1)} = \mathbb{Z}x_{(n, j+1)}$;
	\item for each $n < \omega$, there is $g_n: \omega \rightarrow \omega$ onto such that for $j < \omega$:
	$$h_{(n, j)} := f_{(n+1, n} \restriction G_{(n+1, j)} \in \mrm{Hom}(G_{(n+1, j)}, G_{(n, g_n(j))});$$
	\item for each $n < \omega$, $h_{(n, 0)}$ is the inclusion map from $G^*_{n+1}$ into $G^*_{n+1}$, recalling that $G^*_{n+1} \leq G^*_{n}$;
	\item $g_n(2j+1) = 0$ and $(h_{(n, 2j+1)}(x_{(n+1, 2j+1)}) : j < \omega)$ lists the elements of $G_{(n, 0)}$;
	\item $g_n \restriction \{2j+2 : j < \omega \}$ is $1$-to-$1$ onto $\omega \setminus \{0\}$ and $h_{(n, 2j+2)}(x_{(n+1, 2j+2)}) = x_{g_n(2j+2)}$.
	\end{enumerate}
\end{cclaim}

	\begin{cclaim} For $\mathfrak{s} = (G_n, f_{(n, m)} : m \leq n < \omega)$ and $G = \varprojlim(\mathfrak{s})$ as in \ref{def_inv_comple_dec_count} we have:
	\begin{enumerate}[(1)]
	\item each $G_n$ is completely decomposable (even free);
	\item $G$ is $\mathbb{Z}$-homogeneous;
	\item $G$ is not $\aleph_1$-free.
	\end{enumerate}
\end{cclaim}

	\begin{proof} Item (1) is clear by  \ref{G*n_free} and \ref{def_inv_comple_dec_count}. Concerning (3), let $G_{(\omega, 0)}$ be the subgroup  of $G$ consisting of the elements $(g_n : n < \omega) \in G$ such that for every $n < \omega$, $g_n \in G_{(n, 0)}$. Then $G_{(\omega, 0)} \cong K_{\mathbf{n}}$ (recall \ref{G*n_free}(c)) and so obviously $G$ is not $\aleph_1$-free.
Finally, we show (2), i.e., that $G$ is $\mathbb{Z}$-homogenenous. Let $g = (g_n : n < \omega) \in G \setminus \{0\}$. We distinguish two cases.
\newline \underline{Case 1}. For some $n < \omega$, $g_n \not\in G_{(n, 0)}$.
\newline If this is the case, then for some $n_* < \omega$ we have that, for every $m \geq n_*$, $f_{(\omega, m)}(g) \neq 0$. But then, as in earlier proofs, we can find $h_* \in \mrm{Hom}(G, G_{(n_*, j_*)})$, for some $j_* > 0$, such that $h_*(g) = x_{(n_*, j_*)}$ and so easily $\langle g \rangle^*_G$ is free. 
\newline \underline{Case 2}. For every $n < \omega$, $g_n \in G_{(n, 0)}$.
\newline If this is the case, then $g \in G_{(\omega, 0)}$ and so as $G_{(\omega, 0)}$ is pure in $G$ it suffices to observe that $G_{(\omega, 0)} \cong K_{\mathbf{n}}$ is $\mathbb{Z}$-homogeneous.
\end{proof}

\subsection{For every $\eta \in [T]$, $\bigcap \{L(\eta \restriction n)\}$ is finite is not sufficient}\label{final_subsection}

	In a previous version of this paper we were claiming that if $(T, L)$ is as in \ref{the_engine} and in addition, for every $\eta \in [T]$, $\bigcap \{L(\eta \restriction n)\}$ is finite, then $G(T, L)$ is $\aleph_1$-free. The next example shows that this is {\em not} the case. This motivated the introduction of nice pairs in Section~\ref{sec_nice_pairs}. We include the example for completeness of exposition.

	\begin{cclaim}\label{claim_counterex} Let $(p_n : n \in \omega)$ be distinct primes. Then, for every $n < \omega$, there are $(\bar{p}_n, \bar{a}_n, \bar{Q}_n )$ such that the following conditions hold:
\begin{enumerate}[(a)]
\item $\bar{p}_n = \{p_\eta : \eta \in {2}^n\}$ where $p_\eta \in \{p_\ell : \ell \leq n\}$;
\item $\bar{a}_n = \{a_\eta : \eta \in {2}^n\}$, where $a_\eta \in \mathbb{Z} \setminus \{0\}$;
\item $\bar{Q}_n = \{Q_\eta : \eta \in {2}^n \}$, $Q_\eta \subseteq \{p_\nu : \nu \in 2^n\} \setminus \{p_\eta\}$;
\item $(\prod Q_\eta)$ divides $a_\eta$ in $\mathbb{Z}$ (when $\eta$ is the root of the tree, we let $\prod Q_{\eta} = \{1\}$);
\item $p_\eta$ does not divide $a_\eta$;
\item If $n = m+1$ and $\eta \in 2^m$ then:
	\begin{enumerate}[(i)]
	\item $p_{\eta^\frown (0)} = p_\eta$;
	\item $p_{\eta^\frown (1)} = p_{n+1}$
	\item $\bar{Q}_{\eta^\frown (0)}= \bar{Q}_\eta \cup \{p_{n+1}\}$;
	\item $\bar{Q}_{\eta^\frown (1)} = \bar{Q}_\eta \cup \{p_\eta\}$;
	\item $a_\eta = a_{\eta^\frown (0)} + a_{\eta^\frown (1)}$.
	\end{enumerate}
	\end{enumerate}
\end{cclaim}

\begin{proof} 
\underline{Case 1}. $n = 0$
\newline Let $a_{()} = 1$, $Q_{()} = \emptyset$, $P_{()} = p_0$.
\newline \underline{Case 2}. $n = m+1$
Now $(p_\eta, Q_\eta : \eta \in 2^m)$ are determined by the other conditions, but we have to find for $\eta \in 2^m$, the numbers $a_{\eta^\frown (0)}, a_{\eta^\frown (1)} \in \mathbb{Z} \setminus \{0\}$ satisfying (d)-(f). To this end, choose $a'_0, a'_1 \in \mathbb{Z}$ such that:
\begin{itemize}
	\item $p_\eta \nmid a'_0$ and $p_{n+1} \nmid a'_1$;
	\item $p_\eta a'_0 + p_{n+1} a'_1 = 1$.
\end{itemize}
Let then:
\begin{itemize}
	\item $a_{\eta^\frown (0)} = a_\eta p_\eta a'_0$;
	\item $a_{\eta^\frown (1)} = a_\eta p_{n+1} a'_1$.
\end{itemize}
\end{proof}

	\begin{cclaim} Let $T = 2^{< \omega}$ and, for $\eta \in T$, let $L(\eta) = \{p_\nu : \eta \trianglelefteq \nu \in T\}$. Then $(T, L)$ is as in \ref{the_engine} and we have that there is $y \in G$ such that:
	\begin{enumerate}[(i)]
	\item $f_{(\omega, n)}(y) = \sum \{a_\eta x_\eta : \eta \in 2^n\}$;
	\item in $G$ we have that $\prod_{\ell < n} p_\ell$ divides $y$, for every $n < \omega$;
	\item $G$ is not $\mathbb{Z}$-homogenenous.
	\end{enumerate}
\end{cclaim}

	\begin{proof} Let $y_n = \sum \{a_\eta x_\eta : \eta \in 2^n\}$. Now, clearly $y_n \in G_n$. Furthermore, if $n < m < \omega$, then $f_{(m, n)}(y_m) = y_n$. To see this it suffices to prove it for $n = m +1$. For this purpose, notice:
$$\begin{array}{rcl}
 f_{(n+1, n)}(y_n)& = & f_{(n+1, n)}(\sum \{a_\eta x_\eta : \eta \in 2^{n+1}\}) \\
 & = &  \sum \{a_\eta f_{(n+1, n)}(x_\eta) : \eta \in 2^{n+1}\}\\
 & = &  \sum \{a_\eta x_{\eta \restriction n} : \eta \in 2^{n+1}\}\\
 & = &  \sum \{ \sum \{a_\eta : \nu \triangleleft \eta \in 2^{n+1}\} : \nu \in 2^n\}\\
 & = & \sum \{(a_{\nu^\frown (1)} + a_{\nu^\frown (0)})x_\nu : \nu \in 2^n \},\\
\end{array}$$
but, by \ref{claim_counterex}(f)(v), $a_{\nu^\frown (1)} + a_{\nu^\frown (0)} = a_\nu$ and so we are done.
\end{proof}

\end{document}